\documentclass[11pt]{article}

\usepackage[a4paper,margin=1in]{geometry}
\usepackage{amsmath,amssymb,amsthm,mathtools}
\usepackage{bm}
\usepackage{aliascnt}
\usepackage{enumitem}
\usepackage{microtype}
\usepackage{booktabs}
\usepackage{hyperref}
\usepackage[nameinlink,capitalize,noabbrev]{cleveref}

\hypersetup{
  colorlinks=true,
  linkcolor=blue,
  citecolor=red,
  urlcolor=blue,
  pdftitle={A Critical Density Estimate for the Vlasov--Poisson System: Energy Conservation and Moment Propagation},
  pdfauthor={Zhaopeng Wang}
}

\numberwithin{equation}{section}
\allowdisplaybreaks[2]

\newtheorem{theorem}{Theorem}[section]

\newaliascnt{proposition}{theorem}
\newtheorem{proposition}[proposition]{Proposition}
\aliascntresetthe{proposition}

\newaliascnt{lemma}{theorem}
\newtheorem{lemma}[lemma]{Lemma}
\aliascntresetthe{lemma}

\newaliascnt{corollary}{theorem}
\newtheorem{corollary}[corollary]{Corollary}
\aliascntresetthe{corollary}

\theoremstyle{definition}
\newaliascnt{definition}{theorem}
\newtheorem{definition}[definition]{Definition}
\aliascntresetthe{definition}

\theoremstyle{remark}
\newaliascnt{remark}{theorem}
\newtheorem{remark}[remark]{Remark}
\aliascntresetthe{remark}

\crefname{theorem}{Theorem}{Theorems}
\crefname{proposition}{Proposition}{Propositions}
\crefname{lemma}{Lemma}{Lemmas}
\crefname{corollary}{Corollary}{Corollaries}
\crefname{definition}{Definition}{Definitions}
\crefname{remark}{Remark}{Remarks}

\newcommand{\R}{\mathbb{R}}
\newcommand{\T}{\mathbb{T}}
\newcommand{\M}{\mathcal{M}}
\newcommand{\Kcal}{\mathcal{K}}
\newcommand{\Pcal}{\mathcal{P}}
\newcommand{\Xcal}{\mathcal{X}}
\newcommand{\Ecal}{\mathcal{E}}
\newcommand{\Kstar}{\mathcal{K}_*}
\newcommand{\Div}{\operatorname{Div}}
\newcommand{\diver}{\operatorname{div}}
\newcommand{\Tr}{\operatorname{Tr}}
\newcommand{\tr}{\operatorname{tr}}
\newcommand{\Sym}{\operatorname{Sym}}
\newcommand{\supp}{\operatorname{supp}}
\newcommand{\dd}{\,\mathrm{d}}
\newcommand{\one}{\mathbf{1}}
\newcommand{\Om}{\Omega}
\newcommand{\rhos}{\rho^{\circ}}
\newcommand{\lamOm}{\lambda_{\Om}}
\newcommand{\delOm}{\delta_{\Om}}

\title{A Critical Density Estimate for the Vlasov--Poisson System:\\
Energy Conservation and Moment Propagation}

\author{
Zhaopeng Wang\\[0.4em]
\small School of Mathematics and Statistics\\
\small Beijing Institute of Technology\\
\small Beijing 100081, China\\[0.3em]
\small\texttt{zpwang@bit.edu.cn}
}
\date{}

\begin{document}
\maketitle

\begin{abstract}
This paper establishes a critical space--time density estimate for the
three-dimensional Vlasov--Poisson system in the whole space and on the
torus. This estimate has two consequences. First, every nonnegative weak
solution with finite initial kinetic energy in the bounded finite-energy
weak-solution class conserves the total energy. This gives a positive
answer, within this class, to the energy-conservation question raised by
Lions and Perthame \cite{LionsPerthame1991}. Second, in the periodic repulsive case, it yields
propagation of every velocity moment of order \(k>2\), thereby closing the
previously unresolved interval \(2<k\leq3\).
\end{abstract}



\section{Introduction}

\subsection{The Vlasov-Poisson system}
The Vlasov--Poisson system is a basic mean-field model for the evolution
of a collisionless ensemble of particles. It describes plasmas when the
self-consistent interaction is repulsive and stellar systems when it is
attractive. We consider the system in a spatial domain
\(\Omega\in\{\R^3,\T^3\}\), where
\(\T^3=\R^3/\mathbb Z^3\) is the unit-volume flat torus. The unknown
\(f=f(t,x,v)\geq0\) denotes the phase-space density of particles at time
\(t\in\R^+\), position \(x\in\Omega\), and velocity \(v\in\R^3\). The
initial-value problem reads
\begin{equation}\label{eq:VP-system}
 \begin{aligned}
  &\partial_t f+v\cdot\nabla_xf+E\cdot\nabla_vf=0,
  \qquad f|_{t=0}=f_0,\\
  &\rho(t,x)=\int_{\R^3}f(t,x,v)\dd v,
  \qquad E=-\nabla U,
  \qquad -\Delta U=\lamOm\gamma\rhos.
 \end{aligned}
\end{equation}
Here \(f_0=f_0(x,v)\geq0\) is the prescribed initial phase-space
density, while \(\rho(t,x)\) is the associated spatial particle
density. The functions \(U=U(t,x)\) and \(E=E(t,x)\) denote,
respectively, the self-consistent potential and force field. The parameter \(\gamma\in\{-1,1\}\) specifies the sign of the
interaction. The choice \(\gamma=1\) corresponds to the repulsive plasma
model, whereas \(\gamma=-1\) corresponds to the attractive
gravitational model. The
geometry-dependent coefficient \(\lamOm\) and source density \(\rhos\)
are defined as follows: in the whole-space case,
\(\lamOm=4\pi\) and \(\rhos=\rho\); in the periodic case,
\(\lamOm=1\) and $\rhos=\rho-\bar\rho$,
\[
 \bar\rho
 =\int_{\T^3}\rho(t,x)\dd x
 =\iint_{\T^3\times\R^3}f_0(x,v)\dd x\dd v.
\]
The quantity \(\bar\rho\) is constant in time by conservation of mass. In the whole-space case, the potential is normalized by
\[
 U(t,x)\to0
 \quad\text{as }|x|\to\infty.
\]
In the periodic case, the subtraction of the homogeneous background
\(\bar\rho\) ensures that the right-hand side of the Poisson equation
has zero spatial mean. The potential is uniquely determined by
\[
 \int_{\T^3}U(t,x)\dd x=0.
\]

The energy of the Vlasov--Poisson system is defined by $\Ecal(t):=\Kcal(t)+\Pcal(t)$, where the kinetic and potential energies are defined by
\begin{equation}\label{eq:energy-def}
 \Kcal(t):=\frac12\iint_{\Om\times\R^3}|v|^2f(t,x,v)\dd x\dd v,
 \qquad
 \Pcal(t):=\frac{\gamma}{2\lamOm}\int_{\Om}|E(t,x)|^2\dd x.
\end{equation}

The Cauchy theory for the Vlasov--Poisson system has been extensively
studied; see, for instance, the survey of Rein \cite{Rein2007} for a general
account. Global classical solutions for smooth compactly supported data are
known in both geometries; see
\cite{BattRein1991,LionsPerthame1991,Pfaffelmoser1992,Schaeffer1991}.
The existence of global finite-energy weak solutions in the whole space goes
back to Arsenev and Horst--Hunze \cite{Arsenev1975,HorstHunze1984}. Despite
this classical existence theory, two natural borderline problems remain.

The first concerns the propagation of low velocity moments. Set
\[
  M_k(t):=\iint_{\Om\times\R^3}|v|^kf(t,x,v)\dd x\dd v.
\]
Lions and Perthame proved propagation above order three and asked whether
moments between the energy level and order three are propagated
\cite{LionsPerthame1991}. Pallard later answered this question in the whole
space, proving propagation for every $k>2$, but his periodic argument reached
only $k>14/3$ \cite{Pallard2012}; Chen--Chen subsequently improved the periodic threshold to
$k>3$ \cite{ChenChen2019}. The periodic interval $2<k\leq3$
therefore remained outside the known characteristic-line estimates. Our
first theorem closes this interval and provides, at the same time, a unified
short proof of the low-order estimate in both geometries.

The second question concerns conservation of energy at the natural
finite-energy threshold. The classical equation contains no dissipative mechanism, and its
smooth solutions conserve the total energy.
 By contrast, the standard compactness construction
controls the energy only through weak lower semicontinuity. In the repulsive
case this naturally yields an energy inequality. A
strict loss in the limiting procedure may therefore reflect a defect of weak
compactness rather than a physical dissipation mechanism. Constructing weak solutions that preserve this physical energy structure is consequently important both for the
analysis and for the physical consistency of the model.

In the final discussion of their 1991 paper, Lions and Perthame explicitly
listed conservation of the total energy under only the energy and
$L^\infty$ bounds among the unresolved questions
\cite{LionsPerthame1991}. Ambrosio--Colombo--Figalli later emphasized that,
even after the Lagrangian structure of finite-energy solutions is recovered,
whether the formal conservation law holds for distributional or renormalized
solutions remains an important open problem
\cite{AmbrosioColomboFigalli2017}. The endpoint obstruction is transparent.
If a uniform moment of order $k>2$ is available, then
\[
 \iint_{|v|>R}|v|^2f_n\dd x\dd v
 \leq R^{-(k-2)}M_k(f_n),
\]
so the kinetic-energy tails of a smooth approximation vanish uniformly. One
may then pass the kinetic and field energies separately, using velocity
averaging to obtain strong compactness of the macroscopic density. A related
strategy was used by Wang--Zhang for the periodic Vlasov--Poisson system with
radiation damping, where it yields the physically prescribed energy
dissipation equality \cite{WangZhang2026}. That result relies on a velocity moment strictly above the energy level. At
the second moment, the tail estimate above disappears, and weak convergence
alone may leave an energy defect.

Our second theorem gives a universal energy identity within the bounded
finite-energy weak-solution class. More precisely, every nonnegative weak
solution with finite initial kinetic energy that is essentially bounded in
phase space and has uniformly bounded mass and second velocity moment on
finite time intervals satisfies the critical density estimate and conserves
the total energy. No renormalization or approximation hypothesis is required.
As a consequence, the standard smooth-approximation construction produces a
global renormalized weak solution with this property.

We do not assert that the solution-level bounds appearing in the theorem
follow solely from the initial datum for an arbitrary distributional
solution, and we do not address uniqueness. The pointwise-in-time endpoint
question for the $L^5$ regularity of the field raised in
\cite{LionsPerthame1991} is also different from the present result; the new
regularity obtained here is a space--time estimate.

\subsection{The common critical-density mechanism}

The two borderline problems are resolved by the same space--time estimate.
For a smooth solution, introduce the kinetic mass--momentum tensor
\begin{equation*}
 A(t,x)=
 \begin{pmatrix}
  \rho & j^{\mathsf T}\\
  j & \Pi
 \end{pmatrix},
 \qquad
 j=\int_{\R^3}vf\dd v,
 \qquad
 \Pi=\int_{\R^3}v\otimes vf\dd v.
\end{equation*}
It is positive semidefinite, and the local conservation laws give
\[
  \Div_{t,x}A=\binom{0}{\rho E}.
\]
The source $\rho E$ is integrable using only mass, the $L^\infty$ bound, and
the second moment. For the classical approximations, extension by zero
outside a finite time strip permits the use of Serre's Euclidean
compensated-integrability theorem \cite{Serre2019} on
$\R_t\times\R_x^3$. On $\R_t\times\T_x^3$ one uses Serre's mixed
Euclidean--periodic theorem \cite{Serre2024}; its additional periodic-trace
term is
\[
  \tr\Pi=\int_{\R^3}|v|^2f\dd v,
\]
which is exactly the local kinetic-energy density. Thus the passage from
$\R^3$ to $\T^3$ costs no velocity moment above order two.

For an arbitrary bounded finite-energy weak solution, the same macroscopic
tensor is available after deriving the momentum equation by velocity
truncation. A smooth time cutoff then replaces the endpoint traces used for
the classical approximations and yields the same compensated-integrability
estimate without any approximation hypothesis.

Compensated integrability was previously applied by Serre
\cite{Serre2019} to the repulsive Vlasov--Poisson equation through a
different tensorial construction. In that work, the kinetic
mass--momentum tensor is completed by a nonlocal interaction stress
\(S[\rho]\) so as to obtain a divergence-free positive tensor. The
positivity of this completion is tied to the repulsive sign, and the
resulting estimate controls the nonlocal quantity
\[
    \bigl(\rho\,\det S[\rho]\bigr)^{1/3}
\]
rather than furnishing a Lebesgue estimate for \(\rho\) itself. Here,
instead, the uncompleted local kinetic tensor is retained. Its divergence
is the integrable source
\[
    \begin{pmatrix}
        0\\
        \rho E
    \end{pmatrix},
\]
while its positivity follows solely from \(f\geq0\) and is therefore
independent of the sign of the interaction. The determinant estimate below
then converts compensated integrability directly into
\[
    \rho\in L^2((0,T)\times\Omega),
\]
both in the whole-space and periodic settings. Thus the new point is not
the first use of compensated integrability for the Vlasov--Poisson
equation, but a local Div--BV tensor mechanism that yields a direct
critical density estimate and makes the subsequent energy and moment
arguments possible.

The second ingredient is a sharp pointwise lower bound. If $0\leq g\leq F$,
$m:=\int_{\R^3}g\,\dd v>0$, and
\[
 A_g=\int_{\R^3}\binom{1}{v}\otimes\binom{1}{v}g(v)\dd v,
\]
then
\begin{equation*}
  \det A_g\geq cF^{-2}m^6.
\end{equation*}
We give a detailed proof in \cref{sec:serre}, including the covariance
ellipsoid and the precise Markov-inequality step. Since the space--time
tensor has dimension four, Serre's theorem controls $(\det A)^{1/3}$; the
lower bound above therefore yields
\begin{equation*}
 \rho\in L^2((0,T)\times\Om),
 \qquad
 E\in L^2(0,T;H^1(\Om))\hookrightarrow L^2(0,T;L^6(\Om)).
\end{equation*}
This is the organizing estimate of the paper. For $2<k\leq3$, it places the
local moment $\int|v|^{k-1}f\dd v$ in $L^{6/5}_x$ and closes the differential
inequality for $M_k$. At the critical second moment, it combines with
$j\in L^\infty_tL^{6/5}_x$ to give $E\cdot j\in L^1_{t,x}$. We then derive
the kinetic-energy balance directly for an arbitrary weak solution in the
stated bounded finite-energy class by a monotone velocity truncation; the
potential-energy balance follows from the macroscopic continuity equation
and the Lions--Magenes chain rule.

The contribution is therefore not a juxtaposition of a Euclidean energy
argument and a periodic moment argument. It is a common compensated-density
principle for the Vlasov--Poisson system, with two distinct critical
consequences. In particular, it connects and advances two of the questions
highlighted by Lions and Perthame: it supplies a unified low-order moment
mechanism and reaches the natural finite-energy threshold in the energy
identity.

\subsection{Main results}

Our first theorem concerns moment propagation. The complete $k>2$ statement
is formulated for the repulsive problem, because the higher-order estimates
with which we complete the range are available in that setting. The direct
argument for $2<k\leq3$ is insensitive to the sign once a uniform
finite-energy bound is known.

\begin{theorem}[Moment propagation on the torus]\label{thm:moment-main}
Let $\gamma=1$ and $k>2$. Assume
\[
 0\leq f_0\in L^1(\T^3\times\R^3)\cap L^\infty(\T^3\times\R^3),
 \qquad M_k(0)<\infty.
\]
Then there exists a global nonnegative renormalized weak solution of
\eqref{eq:VP-system}, obtained as a limit of classical solutions, such that
\[
  \sup_{0\leq t\leq T}M_k(t)<\infty
\]
for every $T>0$.
\end{theorem}

Our second theorem addresses conservation of the total energy at the natural
finite-energy threshold. For $T>0$, define the bounded finite-energy space
\begin{equation*}
 \Xcal_T(\Om):=
 \left\{f:\
 \begin{array}{l}
 f\in L^\infty((0,T)\times\Om\times\R^3),\\[0.2em]
 \displaystyle
 \operatorname*{ess\,sup}_{0<t<T}
 \iint_{\Om\times\R^3}(1+|v|^2)|f(t,x,v)|\dd x\dd v<\infty
 \end{array}
 \right\},
\end{equation*}
endowed with
\begin{equation*}
 \|f\|_{\Xcal_T(\Om)}
 :=\|f\|_{L^\infty((0,T)\times\Om\times\R^3)}
 +\operatorname*{ess\,sup}_{0<t<T}
 \iint_{\Om\times\R^3}(1+|v|^2)|f(t,x,v)|\dd x\dd v.
\end{equation*}
We write $f\in\Xcal_{\mathrm{loc}}(\Om)$ when
$f|_{(0,T)}\in\Xcal_T(\Om)$ for every $T>0$.

\begin{theorem}[Universal energy conservation]\label{thm:energy-universal}
Let $\Om\in\{\R^3,\T^3\}$ and $\gamma\in\{-1,1\}$. Let $(f,E)$ be a global
nonnegative weak solution of \eqref{eq:VP-system} with initial datum
\[
 0\leq f_0\in L^1(\Om\times\R^3)\cap L^\infty(\Om\times\R^3),
 \qquad M_2(0)<\infty.
\]
If \(f\in\mathcal X_{\mathrm{loc}}(\Omega)\), then the kinetic and
potential energies have locally absolutely continuous representatives
on \([0,\infty)\), and
\begin{equation}\label{eq:energy-main-conservation}
 \Ecal(t)=\Ecal(0)\qquad\text{for every }t\geq0.
\end{equation}
\end{theorem}

\begin{corollary}[Existence of an energy-conserving solution]
\label{cor:energy-existence}
Let $\Om\in\{\R^3,\T^3\}$ and $\gamma\in\{-1,1\}$. Assume
\[
 0\leq f_0\in L^1(\Om\times\R^3)\cap L^\infty(\Om\times\R^3),
 \qquad M_2(0)<\infty.
\]
Then \eqref{eq:VP-system} admits a global nonnegative renormalized weak
solution, obtained as a limit of classical solutions, which belongs to
$\Xcal_{\mathrm{loc}}(\Om)$ and conserves the total energy at every time.
\end{corollary}

\begin{corollary}[Regularity consequences]
\label{rem:common-regularity}
Every weak solution covered by \cref{thm:energy-universal} satisfies, for
every $T>0$,
\[
 \rho\in L^2((0,T)\times\Om),
 \qquad
 E\in C([0,T];L^2(\Om))\cap L^2(0,T;H^1(\Om)).
\]
The solutions in \cref{thm:moment-main,cor:energy-existence} may moreover be
chosen so that, for every $T>0$ and every $1\leq p<\infty$,
\[
 f\in C([0,T];L^p(\Om\times\R^3))\cap\Xcal_T(\Om).
\]
Here $\Om=\T^3$ for \cref{thm:moment-main}, whereas
$\Om\in\{\R^3,\T^3\}$ for \cref{cor:energy-existence}.
\end{corollary}

\begin{remark}[Scope of the moment-propagation result]
\label{rem:moment-scope}
For $2<k\leq3$, the moment estimate is proved directly in this paper in both
$\R^3$ and $\T^3$. The complete whole-space range is due to
Pallard~\cite{Pallard2012}, while the periodic range $k>3$ is due to
Chen--Chen~\cite{ChenChen2019}. Thus the new moment-propagation conclusion in
\cref{thm:moment-main} is the interval $2<k\leq3$ on the torus.
\end{remark}

\begin{remark}[Meaning of universality]\label{rem:universality}
\Cref{thm:energy-universal} applies to every weak solution belonging to
$\Xcal_{\mathrm{loc}}(\Om)$ with the stated initial datum. It does not assert
that membership in this space follows from the initial datum for an
arbitrary distributional solution, and it does not address uniqueness.
\end{remark}

\paragraph{Notation and conventions.}
Throughout the paper,
\[
 M:=\iint_{\Om\times\R^3}f_0\dd x\dd v,
 \qquad F:=\|f_0\|_{L^\infty(\Om\times\R^3)}.
\]
For the smooth approximations, $\Kstar$ denotes a finite upper bound for
$\sup_n\sup_{t\geq0}\Kcal_n(t)$; its dependence is specified in
\cref{lem:uniform-K}. Unless a domain is displayed, $\|h\|_p$ means the
$L^p(\Om)$ norm, and mixed norms are written as
$\|h\|_{L^q_tL^p_x}$. The letter $C$ denotes a positive constant whose value
may change from line to line. By default it depends only on the dimension;
$C_{\Om}$ may also depend on the fixed periodic geometry, and any dependence
on $T,M,F,$ or $\Kstar$ is indicated when relevant. For an arbitrary weak
solution, dependence on the bounded finite-energy class is recorded through
$\|f\|_{\Xcal_T(\Om)}$.

The paper is organized as follows. \Cref{sec:prelim} records the standard
finite-energy estimates, the macroscopic balance laws, and the approximation
and compactness framework. \Cref{sec:serre} gives the detailed application
of Serre's Euclidean and mixed periodic theorems and proves both the uniform
approximation-level and universal critical density estimates.
\Cref{sec:moments} derives moment propagation. Finally, \cref{sec:energy}
proves the energy-equality criteria, removes the strong time-continuity
assumption, and completes the main results.

\section{Preliminaries}\label{sec:prelim}

\subsection{Weak solutions and finite-energy estimates}

\begin{definition}\label{def:weak}
A nonnegative function $f$ is a weak solution of \eqref{eq:VP-system} on
$[0,\infty)$ if
\[
 f\in L^\infty_{\mathrm{loc}}([0,\infty);L^1(\Om\times\R^3))
 \cap L^\infty_{\mathrm{loc}}([0,\infty)\times\Om\times\R^3),
\]
$\rho=\int f\dd v$ generates $E$ through \eqref{eq:VP-system}, and for every
$\varphi\in C_c^\infty([0,\infty)\times\Om\times\R^3)$,
\begin{equation}\label{eq:weak-form}
 \int_0^\infty\iint_{\Om\times\R^3}
 f\bigl(\partial_t\varphi+v\cdot\nabla_x\varphi
 +E\cdot\nabla_v\varphi\bigr)\dd x\dd v\dd t
 +\iint f_0\varphi(0)\dd x\dd v=0.
\end{equation}
In the periodic case the test functions are periodic in $x$. Such a weak solution is called renormalized if, for every
\(\beta\in C^1([0,\infty))\) satisfying
\(\beta(0)=0\) and \(\beta'\in L^\infty([0,\infty))\), one has
\[
 \partial_t\beta(f)
 +v\cdot\nabla_x\beta(f)
 +E\cdot\nabla_v\beta(f)=0
\]
in the sense of distributions, with initial datum \(\beta(f_0)\).
This is the standard DiPerna--Lions notion of renormalized solution
\cite{DiPernaLions1989}.
\end{definition}

For smooth solutions, integration against $1$ and $v$ gives
\begin{equation}\label{eq:mass-momentum}
 \partial_t\rho+\diver_x j=0,
 \qquad
 \partial_tj+\Div_x\Pi=\rho E.
\end{equation}
Let
\[
  Q_{\Om}=-\nabla(-\Delta_{\Om})^{-1}\diver
\]
be the Helmholtz projection onto gradient fields; on $\T^3$ the inverse is
taken on mean-zero functions. The continuity equation and Poisson's equation
imply
\begin{equation}\label{eq:E-time}
  \partial_tE=-\lamOm\gamma Q_{\Om}j
\end{equation}
in distributions whenever the two sides are defined.

\begin{lemma}[Kinetic interpolation]\label{lem:density-current}
Let $g\geq0$, $g\in L^1(\Om\times\R^3)\cap L^\infty(\Om\times\R^3)$, and set
\[
 m_\ell[g](x):=\int_{\R^3}|v|^\ell g(x,v)\dd v,
 \qquad
 M_\ell(g):=\int_{\Om}m_\ell[g](x)\dd x.
\]
For $0\leq\ell<k$,
\begin{equation}\label{eq:general-kinetic-interpolation}
 m_\ell[g](x)
 \leq C_{k,\ell}\|g\|_\infty^{(k-\ell)/(k+3)}
 m_k[g](x)^{(\ell+3)/(k+3)}.
\end{equation}
Consequently, if $p_{k,\ell}=(k+3)/(\ell+3)$, then
\begin{equation}\label{eq:integrated-kinetic-interpolation}
 \|m_\ell[g]\|_{p_{k,\ell}}
 \leq C_{k,\ell}\|g\|_\infty^{(k-\ell)/(k+3)}
 M_k(g)^{(\ell+3)/(k+3)}.
\end{equation}
Moreover, for $0\leq\ell\leq k$,
\begin{equation}\label{eq:global-moment-interpolation}
 M_\ell(g)\leq M_0(g)^{1-\ell/k}M_k(g)^{\ell/k}.
\end{equation}
\end{lemma}

\begin{proof}
For $R>0$,
\[
 m_\ell[g]\leq C_\ell\|g\|_\infty R^{\ell+3}
                  +R^{\ell-k}m_k[g].
\]
Optimizing in $R$ gives \eqref{eq:general-kinetic-interpolation}; raising
that inequality to the power $p_{k,\ell}$ and integrating in $x$ gives
\eqref{eq:integrated-kinetic-interpolation}. Finally,
\eqref{eq:global-moment-interpolation} follows from H\"older's inequality
with respect to the measure $g\dd x\dd v$.
\end{proof}

\begin{lemma}[Poisson estimates]\label{lem:poisson}
For $1<q<\infty$, the Poisson field satisfies
\[
 \|E\|_2\leq C_{\Om}\|\rhos\|_{6/5},
 \qquad
 \|E\|_{15/4}\leq C_{\Om}\|\rhos\|_{5/3},
 \qquad
 \|\nabla E\|_q\leq C_{\Om,q}\|\rhos\|_q.
\]
In $\R^3$ the first two bounds follow from Hardy--Littlewood--Sobolev and the
third from Calder\'on--Zygmund theory. On $\T^3$ they follow from periodic
elliptic regularity and Poincar\'e's inequality. Moreover, since
$|\T^3|=1$ and $\rho\geq0$,
\[
 \|\rho-\bar\rho\|_{L^p(\T^3)}
 \leq \|\rho\|_{L^p(\T^3)}+\bar\rho
 \leq 2\|\rho\|_{L^p(\T^3)},
 \qquad 1\leq p<\infty.
\]
\end{lemma}

\begin{lemma}[Uniform second-moment bound]\label{lem:uniform-K}
Let $f$ be a smooth solution with mass $M$, $\|f_0\|_\infty\leq F$, and
finite initial second moment. For either sign $\gamma=\pm1$ there is a
constant
\[
 \Kstar=\Kstar(\Om,M,F,\Kcal(0),\|E(0)\|_2)<\infty
\]
such that $\sup_{t\geq0}\Kcal(t)\leq\Kstar$. The same bound is uniform for
the standard smooth approximations of $f_0$.
\end{lemma}

\begin{proof}
Taking $(k,\ell)=(2,0)$ in
\eqref{eq:integrated-kinetic-interpolation},
$\|\rho(t)\|_{5/3}\leq C F^{2/5}\Kcal(t)^{3/5}$. Interpolation between
$L^1$ and $L^{5/3}$, followed by \cref{lem:poisson}, yields
\[
 \|E(t)\|_2^2\leq C_{\Om,M,F}\bigl(1+\Kcal(t)^{1/2}\bigr),
\]
where the harmless additive constant allows the two geometries to be treated simultaneously. If \(\gamma=1\), the positivity of \(\mathcal P\), together with the
classical energy identity
\(\mathcal E(t)=\mathcal E(0)\), immediately controls \(\mathcal K(t)\). If $\gamma=-1$, then
\[
 \Kcal(t)\leq |\Ecal(0)|
   +C_{\Om,M,F}\bigl(1+\Kcal(t)^{1/2}\bigr),
\]
and a quadratic inequality in $\Kcal(t)^{1/2}$ gives the result.

For the approximations introduced below,
\[
 \|\rho_{0,n}-\rho_0\|_1
 \leq \|f_{0,n}-f_0\|_1\longrightarrow0,
\]
while \eqref{eq:integrated-kinetic-interpolation} gives a uniform
$L^{5/3}$ bound for $\rho_{0,n}$. Interpolation therefore yields
$\rho_{0,n}\to\rho_0$ strongly in $L^{6/5}$. In the periodic case,
$\bar\rho_n:=\iint f_{0,n}\to\bar\rho:=\iint f_0$, and hence
\[
 \rho_{0,n}^\circ-\rho_0^\circ
 = (\rho_{0,n}-\rho_0)-(\bar\rho_n-\bar\rho)
 \longrightarrow0
 \quad\text{in }L^{6/5}(\T^3).
\]
The Poisson estimate then gives $E_n(0)\to E(0)$ strongly in $L^2$.
Consequently the same constant $\Kstar$ may be chosen uniformly in $n$.
\end{proof}

\begin{proposition}[Macroscopic balances at finite energy]
\label{prop:macroscopic-balances}
Let $f$ be a weak solution of \eqref{eq:VP-system} on $[0,T]$ with
$f\in\Xcal_T(\Om)$.
Then
\[
 \rho\in L^\infty(0,T;L^1(\Om)\cap L^{5/3}(\Om)),
 \qquad
 j\in L^\infty(0,T;L^1(\Om)\cap L^{5/4}(\Om)),
\]
and
\[
 \Pi\in L^\infty(0,T;L^1(\Om)).
\]
Moreover,
\begin{equation}\label{eq:weak-mass-momentum}
 \partial_t\rho+\diver_xj=0,
 \qquad
 \partial_tj+\Div_x\Pi=\rho E
\end{equation}
in distributions on $(0,T)\times\Om$, with the corresponding initial
traces. In particular,
\[
 \rho\in C([0,T];\mathcal D'(\Om)),
 \qquad \rho(0)=\rho_0,
\]
the total mass is constant, and
\begin{equation}\label{eq:E-time-weak}
 \partial_tE=-\lamOm\gamma Q_{\Om}j
\end{equation}
in distributions.
\end{proposition}

\begin{proof}
The asserted regularity follows from \cref{lem:density-current},
Cauchy--Schwarz with respect to the measure $f\dd x\dd v$, and the second
moment. In particular,
\[
 \|j(t)\|_1
 \leq\left(\iint f(t)\right)^{1/2}
       \left(\iint |v|^2f(t)\right)^{1/2},
 \qquad
 \|\Pi(t)\|_1\leq\iint |v|^2f(t).
\]
By \cref{lem:density-current,lem:poisson},
$\rho\in L^\infty_t(L^1_x\cap L^{5/3}_x)$ and
$E\in L^\infty_t(L^2_x\cap L^{15/4}_x)$. Interpolating $\rho$ between
$L^1$ and $L^{5/3}$ gives
\begin{equation}\label{eq:rho-1511-general}
 \|\rho(t)\|_{15/11}
 \leq \|\rho(t)\|_1^{1/3}\|\rho(t)\|_{5/3}^{2/3},
\end{equation}
so
\begin{equation}\label{eq:rhoE-general-L1}
 \rho E\in L^1((0,T)\times\Om).
\end{equation}

Choose $\chi\in C_c^\infty(\R^3)$ such that $0\leq\chi\leq1$,
$\chi=1$ on $B_1$, and $\supp\chi\subset B_2$, and put
$\chi_R(v)=\chi(v/R)$. Testing \eqref{eq:weak-form} with
$\eta(t,x)\chi_R(v)$, the force-cutoff error is bounded by
\[
 \frac{C}{R}\int_0^T\int_{\Om}|E|\rho\dd x\dd t,
\]
and therefore tends to zero by \eqref{eq:rhoE-general-L1}. The remaining
terms converge by the first-moment bound and dominated convergence. This
proves the continuity equation with initial trace
$\rho_0=\int f_0\dd v$.

Fix $i\in\{1,2,3\}$ and use the test function
\[
 \varphi_R(t,x,v)=\eta(t,x)v_i\chi_R(v).
\]
The time and transport terms converge to the corresponding terms involving
$j_i$ and the $i$th row of $\Pi$. For the force term,
\[
 \nabla_v(v_i\chi_R)=e_i\chi_R+\frac{v_i}{R}\nabla\chi(v/R).
\]
The difference between the first term and $\rho E_i$ is bounded by
$|E|\rho_{>R}$. The second term has the same bound, up to a constant,
because it is supported where $R<|v|<2R$ and $|v_i|/R\leq2$. The second
moment gives
\[
 \|\rho_{>R}\|_{L^\infty(0,T;L^1(\Om))}
 \leq R^{-2}\operatorname*{ess\,sup}_{0<t<T}M_2(t),
\]
while \eqref{eq:integrated-kinetic-interpolation}, applied to
$f\one_{\{|v|>R\}}$, gives an $R$-independent
$L^\infty(0,T;L^{5/3}(\Om))$ bound. Interpolation therefore yields
\[
 \|\rho_{>R}\|_{L^\infty(0,T;L^{15/11}(\Om))}
 \leq C_{\|f\|_{\Xcal_T(\Om)}}R^{-2/3}.
\]
Consequently,
\[
 \int_0^T\int_{\Om}|E|\rho_{>R}\dd x\dd t
 \leq \|E\|_{L^1(0,T;L^{15/4})}
       \|\rho_{>R}\|_{L^\infty(0,T;L^{15/11})}
 \longrightarrow0.
\]
The time term is controlled by $|v|f$, the transport term by $|v|^2f$,
and the initial term by the finite first moment of $f_0$. This proves the
momentum equation and its initial trace.

Mass conservation follows by testing the continuity equation with spatial
cutoffs tending to one; in the whole-space case the boundary term vanishes
because $j\in L^\infty_tL^1_x$. The distributional equations also give the
stated weak time continuity. Finally, differentiating the Poisson equation
and using the continuity equation yields \eqref{eq:E-time-weak}.
\end{proof}

\subsection{Approximation and compactness}

Choose nonnegative smooth data $f_{0,n}$ by truncation and mollification. In
$\R^3$ we truncate in $(x,v)$; on $\T^3$ we truncate only in $v$ and
mollify periodically in $x$. Set
\[
 M_n:=\iint_{\Om\times\R^3}f_{0,n}(x,v)\,\dd x\dd v.
\]
We may arrange
\[
 f_{0,n}\longrightarrow f_0
 \quad\text{in }L^p(\Om\times\R^3),\qquad 1\leq p<\infty,
\]
and
\[
 M_n\leq M,\qquad \|f_{0,n}\|_\infty\leq F.
\]
Whenever $M_k(f_0)<\infty$ for some $k\geq2$, the approximation may
additionally be chosen so that
\[
 M_k(f_{0,n})\longrightarrow M_k(f_0).
\]
In particular, $M_n\to M$.

Let $f_n$ be the corresponding global classical solution and define
\[
 \rho_n(t,x):=\int_{\R^3}f_n(t,x,v)\,\dd v.
\]
In the whole-space case set $\rho_n^\circ:=\rho_n$. In the periodic case set
\[
 \bar\rho_n:=\int_{\T^3}\rho_n(t,x)\,\dd x=M_n,
 \qquad
 \rho_n^\circ:=\rho_n-\bar\rho_n.
\]
The equality $\bar\rho_n=M_n$ follows from conservation of mass. Define
\[
 E_n=-\nabla U_n,
 \qquad
 -\Delta U_n=\lamOm\gamma\rho_n^\circ,
\]
with $U_n\to0$ at infinity in $\R^3$ and
$\int_{\T^3}U_n\,\dd x=0$ on the torus. Thus every periodic approximate
Poisson equation has a mean-zero right-hand side, and
$\bar\rho_n\to\bar\rho$. The standard global classical existence and
continuation theory applies to these smooth whole-space and periodic
problems; the continuation bounds depend on $|E_n|$ and are therefore
unchanged when the sign of $\gamma$ is reversed
\cite{BattRein1991,Pfaffelmoser1992,Schaeffer1991}. Consequently, each
$f_n$ is a global classical solution.

The characteristic flows preserve phase volume, and therefore
\[
 \|f_n(t)\|_1=M_n\leq M,
 \qquad
 \|f_n(t)\|_\infty\leq F.
\]
The uniform kinetic-energy bound follows separately from
\cref{lem:uniform-K}:
\[
 \sup_n\sup_{t\geq0}\Kcal_n(t)\leq\Kstar.
\]
By \eqref{eq:integrated-kinetic-interpolation}, \cref{lem:poisson}, and the
centered-density estimate in that lemma, for every $T>0$,
\[
 \sup_n\|\rho_n\|_{L^\infty(0,T;L^1(\Om)\cap L^{5/3}(\Om))}<\infty,
 \qquad
 \sup_n\|E_n\|_{L^\infty(0,T;L^2(\Om)\cap L^{15/4}(\Om))}<\infty.
\]

The compactness construction uses two standard inputs, which we record in
the form needed below.

\begin{lemma}[Spatial tightness in the whole space]\label{lem:spatial-tightness}
Assume $\Om=\R^3$. There exists an increasing function
$0\leq G\in C^1([0,\infty))$ such that $0\leq G'\leq1$, $G(R)\to\infty$, and, for
every $T>0$,
\[
 \sup_n\sup_{0\leq t\leq T}
 \iint_{\R^3\times\R^3}G(|x|)f_n(t,x,v)\dd x\dd v<\infty.
\]
Consequently,
\[
 \lim_{R\to\infty}\sup_n
 \int_0^T\int_{|x|>R}\rho_n(t,x)\dd x\dd t=0.
\]
\end{lemma}

\begin{proof}
Let \(G\) be the function furnished by the slowly increasing-weight
lemma of Horst and Hunze \cite[Lemma~4.4]{HorstHunze1984} with
$\int_{\R^3}G(|x|)\rho_0(x)\dd x<\infty$. Let $\eta_n^x$ denote the spatial mollifier, chosen with
$\supp\eta_n^x\subset B_{1/n}$. Since $G$ is one-Lipschitz,
\begin{align}
 \int_{\R^3}G(|x|)\rho_{0,n}(x)\dd x&\leq \iint_{\R^3\times\R^3}G(|y+z|)\eta_n^x(z)\rho_0(y)\dd z\dd y\nonumber\\
 &\leq \int_{\R^3}\bigl(G(|y|)+1\bigr)\rho_0(y)\dd y,\nonumber
\end{align}
so the weighted initial data are uniformly controlled. Since
$x\mapsto G(|x|)$ is one-Lipschitz, its composition with every absolutely
continuous characteristic is absolutely continuous and satisfies
\[
 \left|\frac{\dd}{\dd t}G(|X_n(t)|)\right|\leq |V_n(t)|
\]
for almost every time. Consequently,
\[
 \frac{\dd}{\dd t}\iint G(|x|)f_n(t,x,v)\dd x\dd v
 \leq\iint |v|f_n(t,x,v)\dd x\dd v
 \leq M^{1/2}(2\Kstar)^{1/2}
\]
for almost every $t$.
Integrating from $0$ to $t\leq T$ proves the first assertion. On
$\{|x|>R\}$ one has $G(|x|)\geq G(R)$; hence
\[
 \int_0^T\int_{|x|>R}\rho_n\dd x\dd t
 \leq \frac{T}{G(R)}
 \sup_{0\leq t\leq T}\iint G(|x|)f_n(t,x,v)\dd x\dd v,
\]
which proves the spatial-tail estimate.
\end{proof}

\begin{lemma}[Velocity averaging]\label{lem:velocity-averaging}
Let $\Om\in\{\R^3,\T^3\}$ and let
$h,g,H\in L^2((0,T)\times\Om\times\R^3)$ satisfy
\[
 \partial_th+v\cdot\nabla_xh=g+\diver_vH
\]
in distributions. Then, for every $\psi\in C_c^\infty(\R^3)$,
\[
 \int_{\R^3}h(t,x,v)\psi(v)\dd v
 \in L^2(0,T;H^{1/4}(\Om)),
\]
with a bound depending only on $\psi$ and the three $L^2$ norms. The
Euclidean statement is due to
Bouchut--Desvillettes and Golse--Lions--Perthame--Sentis; the periodic
version follows by Fourier series or periodic localization
\cite{BouchutDesvillettes1999,GLPS1988}.
\end{lemma}

The following compactness statement collects mostly standard consequences
of the finite-energy approximation framework. For the reader's convenience,
we include a brief proof, keeping track of the precise convergence properties
needed later.

\begin{proposition}[Finite-energy compactness]\label{prop:compactness}
After extraction of a subsequence and a diagonal argument in $T$, there is a
global nonnegative renormalized weak solution $(f,E)$ with the following
properties on every finite interval $[0,T]$:
\begin{enumerate}[label=\textup{(\roman*)},leftmargin=2.6em]
\item The phase-space densities converge as
\[
 f_n\rightharpoonup f
 \quad\text{in }L^1((0,T)\times\Om\times\R^3),
 \qquad
 f_n\stackrel{*}{\rightharpoonup}f
 \quad\text{in }L^\infty((0,T)\times\Om\times\R^3).
\]

\item The spatial densities satisfy
\begin{align}
 \rho_n&\to\rho\quad
 \text{strongly in }L^1((0,T)\times\Om),
 \label{eq:rho-strong1}\\
 \rho_n&\to\rho\quad
 \text{strongly in }L^2(0,T;L^{6/5}(\Om)).
 \label{eq:rho-mixed}
\end{align}

\item The fields satisfy
\begin{equation}\label{eq:E-strong2}
 E_n\to E
 \quad\text{strongly in }L^2((0,T)\times\Om).
\end{equation}

\item The limit obeys
\begin{equation}\label{eq:limit-M2}
 \begin{gathered}
  \|f(t)\|_{L^1(\Om\times\R^3)}=M,\qquad
  \operatorname*{ess\,sup}_{0<t<T}M_2(t)\leq2\Kstar,
 \end{gathered}
\end{equation}
and
\begin{equation}\label{eq:f-strong-time}
 f\in C([0,T];L^p(\Om\times\R^3)),
 \qquad1\leq p<\infty.
\end{equation}
\end{enumerate}
\end{proposition}

\begin{proof}
Fix $T>0$. All subsequences below are extracted on $[0,T]$; a diagonal
argument over integer values of $T$ will be used at the end.

\emph{Step 1: weak compactness in phase space.}
The second-moment estimate gives
\[
 \int_0^T\iint_{|v|>R}f_n\,\dd x\dd v\dd t
 \leq \frac{2T\Kstar}{R^2}.
\]
In the whole-space case, \cref{lem:spatial-tightness} also controls the
spatial tails; on the torus no spatial-tail estimate is needed. Together
with $0\leq f_n\leq F$, these bounds give uniform integrability and
tightness on $(0,T)\times\Om\times\R^3$. By the Dunford--Pettis theorem and
Banach--Alaoglu, after extraction,
\[
 f_n\rightharpoonup f\quad\text{in }L^1,
 \qquad
 f_n\stackrel{*}{\rightharpoonup}f\quad\text{in }L^\infty,
\]
with $0\leq f\leq F$.

\emph{Step 2: compactness of truncated velocity averages.}
Fix $R>0$, choose $\chi_R\in C_c^\infty(\R^3)$ with
$0\leq\chi_R\leq1$, $\chi_R=1$ on $B_R$, and
$\supp\chi_R\subset B_{2R}$, and set $h_{n,R}=\chi_Rf_n$. Then
\[
 \partial_th_{n,R}+v\cdot\nabla_xh_{n,R}
 =-\diver_v(E_nh_{n,R})+E_nf_n\cdot\nabla_v\chi_R.
\]
For fixed $R$,
\[
 \|h_{n,R}\|_{L^2_{t,x,v}}^2\leq FTM,
\]
and, using the uniform $L^2_{t,x}$ bound for $E_n$,
\[
 \|E_nh_{n,R}\|_{L^2_{t,x,v}}
 +\|E_nf_n\cdot\nabla_v\chi_R\|_{L^2_{t,x,v}}
 \leq C_RF\|E_n\|_{L^2_{t,x}}.
\]
Hence \cref{lem:velocity-averaging} gives
\[
 \rho_{n,R}:=\int_{\R^3}\chi_R(v)f_n(t,x,v)\,\dd v
 \quad\text{bounded in }L^2(0,T;H^{1/4}(\Om)).
\]
Moreover, with
\[
 J_{n,R}:=\int_{\R^3}v\chi_Rf_n\,\dd v,
 \qquad
 S_{n,R}:=E_n\cdot\int_{\R^3}f_n\nabla_v\chi_R\,\dd v,
\]
we have
\[
 \partial_t\rho_{n,R}+\diver_xJ_{n,R}=S_{n,R}.
\]
The fixed velocity support, the mass bound, and $f_n\leq F$ imply
$J_{n,R}$ is uniformly bounded in $L^2((0,T)\times\Om)$, while
$S_{n,R}$ is uniformly bounded in the same space. Thus
$\partial_t\rho_{n,R}$ is bounded in $L^2(0,T;H^{-1}(\Om))$. Aubin--Lions
therefore yields strong compactness of $\rho_{n,R}$ in
$L^2((0,T)\times\T^3)$ on the torus and in
$L^2((0,T)\times K)$ for every compact $K\Subset\R^3$ in the whole-space
case. The weak convergence of $f_n$ identifies the strong limit as
\[
 \rho_R(t,x):=\int_{\R^3}\chi_R(v)f(t,x,v)\,\dd v.
\]

\emph{Step 3: removal of the velocity and spatial cutoffs.}
Weak lower semicontinuity, applied first to bounded velocity truncations,
gives
\[
 \int_0^TM_2(f(t))\,\dd t\leq2T\Kstar.
\]
Consequently,
\[
 \|\rho_n-\rho_{n,R}\|_{L^1((0,T)\times\Om)}
 +\|\rho-\rho_R\|_{L^1((0,T)\times\Om)}
 \leq \frac{4T\Kstar}{R^2}.
\]
On the torus this, together with the strong convergence of $\rho_{n,R}$,
gives \eqref{eq:rho-strong1}. In $\R^3$, one first works on a fixed ball
and then removes the spatial cutoff by \cref{lem:spatial-tightness}; the
same spatial-tail bound passes to $f$ by lower semicontinuity. We again
obtain \eqref{eq:rho-strong1}, and the construction identifies
$\rho=\int f\,\dd v$. Since $\int_{\Om}\rho_n(t,x)\,\dd x=M_n$ for
every $t$,
\[
 \int_0^T
 \left|\int_{\Om}\rho(t,x)\,\dd x-M\right|\dd t
 \leq \|\rho_n-\rho\|_{L^1((0,T)\times\Om)}
      +T|M_n-M|\longrightarrow0.
\]
Hence $\int_{\Om}\rho(t,x)\,\dd x=M$ for almost every $t$, and
$f\in L^\infty(0,T;L^1(\Om\times\R^3))$.

\emph{Step 4: stronger density convergence and convergence of the fields.}
Let $r_n=\rho_n-\rho$. The uniform $L^\infty_tL^{5/3}_x$ bounds and
interpolation give
\[
 \|r_n(t)\|_{6/5}^2
 \leq \|r_n(t)\|_1^{7/6}\|r_n(t)\|_{5/3}^{5/6}
 \leq C\|r_n(t)\|_1^{7/6}.
\]
Since $\|r_n(t)\|_1$ is uniformly bounded,
\[
 \int_0^T\|r_n(t)\|_{6/5}^2\,\dd t
 \leq C\int_0^T\|r_n(t)\|_1\,\dd t\longrightarrow0,
\]
which proves \eqref{eq:rho-mixed}. In the whole-space case this immediately
implies \eqref{eq:E-strong2}. On the torus,
\[
 \rho_n^\circ-\rho^\circ
 = (\rho_n-\rho)-(\bar\rho_n-\bar\rho),
\]
and hence
\[
 \|\rho_n^\circ-\rho^\circ\|_{L^2(0,T;L^{6/5})}
 \leq \|\rho_n-\rho\|_{L^2(0,T;L^{6/5})}
      +T^{1/2}|\bar\rho_n-\bar\rho|\longrightarrow0.
\]
The Poisson estimate proves \eqref{eq:E-strong2} in the periodic case as
well.

\emph{Step 5: the limit equation, the second moment, and renormalization.}
For a compactly supported test function $\varphi$, set
\[
 H_n(t,x):=\int_{\R^3}f_n(t,x,v)\nabla_v\varphi(t,x,v)\,\dd v.
\]
Then $H_n\rightharpoonup H:=\int f\nabla_v\varphi\,\dd v$ weakly in
$L^2_{t,x}$, while $E_n\to E$ strongly in $L^2_{t,x}$. Thus
$\int E_n\cdot H_n\to\int E\cdot H$, and all terms in the weak formulation
pass to the limit, including the initial term because $f_{0,n}\to f_0$ in
$L^1$.

For every bounded nonnegative velocity truncation $\beta_R\uparrow|v|^2$,
the functions $t\mapsto\iint\beta_Rf_n(t)$ are bounded in
$L^\infty(0,T)$ by $2\Kstar$. Passing to the weak-* limit in time and then
letting $R\to\infty$ gives the second-moment bound in
\eqref{eq:limit-M2}.

The uniform $L^\infty_tL^{5/3}_x$ bound for $\rho_n$ and the
uniform $L^\infty_t(L^2_x\cap L^{15/4}_x)$ bound for $E_n$ pass to the
limit by weak-* compactness; the strong convergences already proved identify
the weak-* limits with $\rho$ and $E$. Consequently,
\[
 \nabla_xE\in L^\infty(0,T;L^{5/3}(\Om)),
 \qquad
 E\in L^\infty(0,T;L^2(\Om)\cap L^{15/4}(\Om))
 \subset L^\infty(0,T;L^3(\Om)).
\]
Hence the phase-space vector field $b=(v,E)$ belongs to
$L^1(0,T;W^{1,1}_{\mathrm{loc}})$ and is divergence free. In the
whole-space case its global growth condition also holds. Indeed, the
velocity component of $b/(1+|x|+|v|)$ is bounded, while, with
$h=|E(x)|/(1+|x|+|v|)$,
\[
 h\one_{\{h\leq1\}}\in L^\infty,
 \qquad
 \iint h\one_{\{h>1\}}\,\dd x\dd v
 \leq C\int_{\R^3}|E(x)|^3\,\dd x<\infty.
\]
The DiPerna--Lions renormalization theorem therefore applies to
$b=(v,E)$. Since $b$ is divergence free, the renormalized equation preserves
the $L^p$ norms of $f$, and the associated continuity theorem gives
\eqref{eq:f-strong-time}. In particular, the mass identity obtained above
for almost every time holds for every $t\in[0,T]$, which completes
\eqref{eq:limit-M2}. A diagonal extraction over $T\in\mathbb N$ completes
the construction of the global solution.
\end{proof}

\begin{lemma}[Time-slice lower semicontinuity]\label{lem:time-slice}
After a further subsequence, for every $t\in[0,T]$,
\[
  M_k(t)\leq\liminf_{n\to\infty}M_k(f_n(t))
\]
whenever the right-hand side is uniformly finite on $[0,T]$.
\end{lemma}

\begin{proof}
For every $\phi\in C_c^\infty(\Om\times\R^3)$, the classical
Vlasov equation gives
\[
 \frac{\dd}{\dd t}\iint f_n(t)\phi
 =\iint f_n(t)\bigl(v\cdot\nabla_x\phi
     +E_n(t,x)\cdot\nabla_v\phi\bigr).
\]
The right-hand side is uniformly bounded on $[0,T]$, because $\phi$ is
compactly supported, $f_n\leq F$, and $E_n$ is bounded in
$L^\infty(0,T;L^2)$. Thus the time pairings form an equicontinuous family.
Arzel\`a--Ascoli and a diagonal argument yield a continuous limit for
each compactly supported pairing. For almost every time, this limit agrees
with the spacetime weak limit obtained in \cref{prop:compactness}; since
$f\in C([0,T];L^1)$, the two representatives agree for every time. Hence
\[
 f_n(t)\rightharpoonup f(t)
\]
against compactly supported test functions for every $t\in[0,T]$.

By a standard smooth monotone approximation from below, choose
nonnegative functions
\[
 w_R\in C_c^\infty(\Om\times\R^3),
 \qquad
 0\leq w_R\uparrow |v|^k.
\]
No spatial cutoff is needed on the torus. Then, for every fixed $R$,
\[
 \iint w_R f(t)
 =
 \lim_{n\to\infty}\iint w_R f_n(t)
 \leq
 \liminf_{n\to\infty}M_k(f_n(t)).
\]
Letting \(R\to\infty\) and using monotone convergence gives
\[
 M_k(t)\leq\liminf_{n\to\infty}M_k(f_n(t)).
\]
\end{proof}

\section{The critical density estimate}\label{sec:serre}

\subsection{The two compensated-integrability theorems}

For an $N\times N$ symmetric matrix-valued finite Radon measure $A$, the
row-wise divergence is
\[
 (\Div A)_i=\sum_{j=1}^N\partial_jA_{ij}.
\]
We call $A$ a positive Div--BV tensor if $A$ is positive semidefinite as a
matrix-valued measure and $\Div A$ is a finite vector-valued measure.  In our
application $A$ has an $L^1$ density, so its determinant is the ordinary
pointwise determinant.

The Euclidean theorem is the following consequence of Serre's compensated
integrability theory \cite{Serre2018,Serre2019}.

\begin{theorem}[Euclidean case]\label{thm:serre-euclidean}
Let $A\in L^1(\R^N;\Sym_N^+)$ and
$\Div A\in\M(\R^N;\R^N)$.  Then
$(\det A)^{1/N}\in L^{N/(N-1)}(\R^N)$ and
\begin{equation*}
 \int_{\R^N}(\det A)^{1/(N-1)}\dd z
 \leq C_N\|\Div A\|_{\M(\R^N)}^{N/(N-1)}.
\end{equation*}
\end{theorem}

For a mixed Euclidean--periodic domain, Serre proved the following theorem
\cite{Serre2024}.

\begin{theorem}[mixed periodic case]\label{thm:serre-periodic}
Let $N=k+m$, and let $A$ be a positive Div--BV tensor on
$\R^k\times\T^m$.  If
\[
 \Tr_mA=\sum_{j=k+1}^{N}A_{jj}
\]
denotes the trace in the periodic directions, then
\begin{equation}\label{eq:serre-periodic}
 \int_{\R^k\times\T^m}(\det A)^{1/(N-1)}\dd z
 \leq C_{k,m,\T^m}
 \bigl(\|\Div A\|_{\M}+\|\Tr_mA\|_{\M}\bigr)^{N/(N-1)}.
\end{equation}
\end{theorem}

Serre states the theorem first on the standard torus and then obtains an
arbitrary flat torus by a linear congruence transformation.  Thus
\eqref{eq:serre-periodic} applies to $\T^3=\R^3/\mathbb Z^3$, with a
geometry-dependent constant.

For the present four-dimensional space--time tensor, the two statements may
be written in the common form
\begin{equation}\label{eq:serre-unified}
 \int_{\R\times\Om}(\det A)^{1/3}\dd t\dd x
 \leq C_{\Om}
 \bigl(\|\Div A\|_{\M}
       +\delOm\|\Tr_xA\|_{\M}\bigr)^{4/3},
\end{equation}
where
\[
 \delOm=
 \begin{cases}0,&\Om=\R^3,\\1,&\Om=\T^3.\end{cases}
\]
For $\Om=\R^3$, \eqref{eq:serre-unified} is
\cref{thm:serre-euclidean}; for $\Om=\T^3$ it is
\cref{thm:serre-periodic} with $k=1$ and $m=3$.  It is important not to use
the fully periodic Div-free theorem here: time is noncompact, the kinetic
tensor has a nonzero divergence, and the periodic version necessarily
contains the additional spatial-trace term.

\subsection{The approximation tensors as Div--BV tensors}

For a smooth approximation define
\begin{equation*}
 A_n(t,x)=
 \begin{pmatrix}
  \rho_n & j_n^{\mathsf T}\\
  j_n & \Pi_n
 \end{pmatrix}
 =\int_{\R^3}\binom{1}{v}\otimes\binom{1}{v}f_n(t,x,v)\dd v.
\end{equation*}
For every $\zeta=(\zeta_0,\zeta')\in\R^4$,
\begin{equation*}
 \zeta^{\mathsf T}A_n\zeta
 =\int_{\R^3}(\zeta_0+\zeta'\cdot v)^2f_n\dd v\geq0,
\end{equation*}
so $A_n\in\Sym_4^+$.  By \eqref{eq:mass-momentum},
\begin{equation*}
 \Div_{t,x}A_n=\binom{0}{\rho_nE_n}.
\end{equation*}

We first verify that the source is controlled before using the desired
$L^2$ estimate for $\rho_n$.  Interpolation between $L^1$ and $L^{5/3}$ gives
\begin{equation*}
 \|\rho_n(t)\|_{15/11}
 \leq M^{1/3}\|\rho_n(t)\|_{5/3}^{2/3}.
\end{equation*}
By \cref{lem:poisson} and H\"older's inequality,
\begin{align}
 \|\rho_n(t)E_n(t)\|_1
 &\leq \|\rho_n(t)\|_{15/11}\|E_n(t)\|_{15/4}\notag\\
 &\leq C_{\Om}M^{1/3}\|\rho_n(t)\|_{5/3}^{5/3}
\notag\\
 &\leq C_{\Om,M,F,\Kstar}.
 \label{eq:rhoE-L1}
\end{align}
Here $15/11$ and $15/4$ are conjugate exponents.  This estimate uses only the
finite-energy layer and is therefore not circular.

Fix $T>0$ and extend the tensor by zero:
\[
 A_n^e(t,x)=\one_{(0,T)}(t)A_n(t,x).
\]
In distributions,
\begin{equation}\label{eq:zero-extension-div}
 \Div A_n^e
 =\one_{(0,T)}\binom{0}{\rho_nE_n}
 +\binom{\rho_n(0)}{j_n(0)}\delta_{t=0}
 -\binom{\rho_n(T)}{j_n(T)}\delta_{t=T}.
\end{equation}
The endpoint currents satisfy
\begin{equation}\label{eq:boundary-current}
 \|j_n(t)\|_1
 \leq\iint |v|f_n\dd x\dd v
 \leq M^{1/2}(2\Kstar)^{1/2}.
\end{equation}
Since a positive semidefinite matrix satisfies
$|A_{ij}|\leq(A_{ii}+A_{jj})/2$,
\begin{equation*}
 \int_0^T\int_{\Om}|A_n|\dd x\dd t
 \leq C\int_0^T\bigl(M+2\Kcal_n(t)\bigr)\dd t\leq C_T.
\end{equation*}
Thus $A_n^e$ is a finite positive tensor and, by
\eqref{eq:rhoE-L1}--\eqref{eq:boundary-current},
\begin{equation*}
 \|\Div A_n^e\|_{\M(\R\times\Om)}\leq C(1+T).
\end{equation*}
In the periodic case, the trace in the three periodic directions is
\begin{equation*}
 \Tr_xA_n^e
 =\one_{(0,T)}\tr\Pi_n
 =\one_{(0,T)}\int_{\R^3}|v|^2f_n\dd v,
\end{equation*}
and hence
\begin{equation*}
 \|\Tr_xA_n^e\|_{\M(\R\times\T^3)}
 =\int_0^TM_2(f_n(t))\dd t\leq2\Kstar T.
\end{equation*}
Applying \eqref{eq:serre-unified} yields
\begin{equation}\label{eq:det-serre-bound}
 \sup_n\int_0^T\int_{\Om}(\det A_n)^{1/3}\dd x\dd t
 \leq C_{\Om,M,F,\Kstar}(1+T)^{4/3}.
\end{equation}

\subsection{The covariance determinant}

We now prove the algebraic estimate that converts
\eqref{eq:det-serre-bound} into density integrability.

\begin{lemma}[Kinetic determinant lower bound]\label{lem:determinant}
Let $0\leq g\in L^1(\R^3)\cap L^\infty(\R^3)$ and
$|v|^2g\in L^1(\R^3)$.  Set
\[
 m=\int_{\R^3}g(v)\dd v,
 \qquad
 A_g=\int_{\R^3}\binom{1}{v}\otimes\binom{1}{v}g(v)\dd v.
\]
If $m=0$, then $g=0$ almost everywhere and $\det A_g=0$. If $m>0$, then
\begin{equation}\label{eq:determinant-lower}
 \det A_g\geq c\|g\|_\infty^{-2}m^6,
\end{equation}
where $c>0$ is universal.
\end{lemma}

\begin{proof}
Assume $m>0$ and define the mean velocity
and covariance matrix
\begin{equation*}
 u=\frac1m\int_{\R^3}vg(v)\dd v,
 \qquad
 C=\frac1m\int_{\R^3}(v-u)\otimes(v-u)g(v)\dd v.
\end{equation*}
The matrix $C$ is positive definite.  Indeed, if
$\xi^{\mathsf T}C\xi=0$ for some $\xi\neq0$, then
\[
 \int_{\R^3}|\xi\cdot(v-u)|^2g(v)\dd v=0.
\]
Thus $g$ vanishes almost everywhere outside the affine hyperplane
$\{v:\xi\cdot(v-u)=0\}$.  This hyperplane has three-dimensional Lebesgue
measure zero, contradicting $m>0$ because $g$ is a Lebesgue density.

Consider the ellipsoid
\begin{equation*}
 \mathcal E=
 \bigl\{v\in\R^3:(v-u)^{\mathsf T}C^{-1}(v-u)\leq6\bigr\}.
\end{equation*}
Let $\mu$ be the probability measure $\dd\mu=g(v)\dd v/m$, and define the
nonnegative random variable
\[
 X(v)=(v-u)^{\mathsf T}C^{-1}(v-u).
\]
By the definition of $C$,
\begin{align*}
 \int_{\R^3}X(v)\dd\mu(v)
 &=\frac1m\int_{\R^3}
   \tr\bigl(C^{-1}(v-u)\otimes(v-u)\bigr)g(v)\dd v\notag\\
 &=\tr\left(C^{-1}\frac1m\int_{\R^3}
 (v-u)\otimes(v-u)g(v)\dd v\right)\notag\\
 &=\tr(C^{-1}C)=3.
\end{align*}
Markov's inequality for the nonnegative random variable $X$ gives
\begin{equation*}
 \mu(\mathcal E^c)=\mu(X>6)
 \leq\frac{\int X\dd\mu}{6}=\frac12.
\end{equation*}
Consequently,
\begin{equation*}
 \int_{\mathcal E}g(v)\dd v=m\mu(\mathcal E)\geq\frac m2.
\end{equation*}
On the other hand, $g\leq\|g\|_\infty$, and the change of variables
$v=u+C^{1/2}y$ yields
\begin{equation*}
 |\mathcal E|
 =\det(C^{1/2})\,|B_{\sqrt6}|
 =|B_{\sqrt6}|\sqrt{\det C}.
\end{equation*}
Therefore
\begin{equation*}
 \frac m2
 \leq\int_{\mathcal E}g(v)\dd v
 \leq\|g\|_\infty|B_{\sqrt6}|\sqrt{\det C}.
\end{equation*}
It follows that
\begin{equation}\label{eq:covariance-det-lower}
 \det C\geq c\frac{m^2}{\|g\|_\infty^2}.
\end{equation}

It remains to relate $C$ to $A_g$.  Since
\[
 \int v\otimes vg\dd v=m\,u\otimes u+mC,
\]
we have
\[
 A_g=
 \begin{pmatrix}
  m & mu^{\mathsf T}\\
  mu & m\,u\otimes u+mC
 \end{pmatrix}.
\]
The Schur complement of the upper-left entry $m$ is $mC$, and hence
\begin{equation}\label{eq:schur-determinant}
 \det A_g=m\det(mC)=m^4\det C.
\end{equation}
Combining \eqref{eq:covariance-det-lower} and
\eqref{eq:schur-determinant} proves
\eqref{eq:determinant-lower}.
\end{proof}

\subsection{The critical density estimate for the approximations}

If $F=0$, then $f_0=0$ and all conclusions are trivial. Assume $F>0$.
At points where $\rho_n(t,x)=0$, the desired inequality is immediate. At
points where $\rho_n(t,x)>0$, apply \cref{lem:determinant} to
$g(v)=f_n(t,x,v)$ and use $\|f_n\|_\infty\leq F$. Thus, almost everywhere,
\begin{equation*}
 \det A_n(t,x)\geq cF^{-2}\rho_n(t,x)^6,
 \qquad
 \rho_n(t,x)^2\leq CF^{2/3}(\det A_n(t,x))^{1/3}.
\end{equation*}
Together with \eqref{eq:det-serre-bound}, this proves the central estimate.

\begin{theorem}[Critical density estimate for the approximation scheme]\label{thm:critical-density-approx}
For every $T>0$,
\begin{equation}\label{eq:rho-L2-approx}
 \sup_n\int_0^T\int_{\Om}\rho_n(t,x)^2\dd x\dd t
 \leq C_{\Om,M,F,\Kstar}(1+T)^{4/3}.
\end{equation}
The limiting weak solution of \cref{prop:compactness} satisfies
\begin{equation*}
 \rho\in L^2((0,T)\times\Om)
\end{equation*}
and
\begin{equation}\label{eq:E-critical}
 E\in L^2(0,T;H^1(\Om))
 \hookrightarrow L^2(0,T;L^6(\Om)).
\end{equation}
\end{theorem}

\begin{proof}
Only the limit passage remains.  By \eqref{eq:rho-L2-approx}, a subsequence
converges weakly in $L^2((0,T)\times\Om)$.  The strong $L^1$ convergence
\eqref{eq:rho-strong1} identifies its weak limit with $\rho$.  Finally,
\cref{lem:poisson} gives $\nabla E\in L^2_{t,x}$; $E\in L^\infty_tL^2_x$
comes from the finite-energy layer.  This proves \eqref{eq:E-critical}.
\end{proof}

\subsection{The universal critical density estimate}

The preceding argument is tailored to the classical approximations and gives
the uniform estimate required in \cref{sec:moments}. We now show that the
same critical density bound is intrinsic to every weak solution in the
bounded finite-energy class. The only additional issue is the absence of
strong endpoint traces for $(\rho,j)$, which is resolved by a smooth time
cutoff.

Let $(f,E)$ be a weak solution on $[0,T]$ with
$f\in\Xcal_T(\Om)$, and define
\begin{equation}\label{eq:universal-tensor}
 A(t,x)=
 \begin{pmatrix}
  \rho & j^{\mathsf T}\\
  j & \Pi
 \end{pmatrix}
 =\int_{\R^3}\binom{1}{v}\otimes\binom{1}{v}f(t,x,v)\dd v.
\end{equation}
Then $A\geq0$, $A\in L^1((0,T)\times\Om)$, and
\begin{equation}\label{eq:universal-tensor-div}
 \Div_{t,x}A=\binom{0}{\rho E}
\end{equation}
in distributions. Notice that the source is already integrable by
\eqref{eq:rhoE-general-L1}; no $L^2_{t,x}$ estimate for $\rho$ has been
used.

\begin{lemma}[Time-cutoff Div--BV extension]
\label{lem:time-cutoff-divbv}
There exist functions $\eta_\varepsilon\in C_c^\infty((0,T))$ satisfying
\[
 0\leq\eta_\varepsilon\leq1,
 \qquad
 \eta_\varepsilon(t)\longrightarrow1\quad(t\in(0,T)),
 \qquad
 \int_0^T|\eta_\varepsilon'(t)|\dd t\leq C,
\]
such that, after extension by zero outside $(0,T)$,
\[
 B_\varepsilon(t,x):=\eta_\varepsilon(t)A(t,x)
\]
is a positive Div--BV tensor on $\R\times\Om$ and
\begin{equation}\label{eq:time-cutoff-div}
 \Div B_\varepsilon
 =\eta_\varepsilon\binom{0}{\rho E}
 +\eta_\varepsilon'\binom{\rho}{j}.
\end{equation}
Moreover,
\begin{equation}\label{eq:time-cutoff-div-bound}
 \sup_{\varepsilon>0}
 \|\Div B_\varepsilon\|_{\M(\R\times\Om)}
 \leq C_{\Om,T,\|f\|_{\Xcal_T(\Om)}},
\end{equation}
and, in the periodic case,
\begin{equation}\label{eq:time-cutoff-trace-bound}
 \sup_{\varepsilon>0}
 \|\Tr_xB_\varepsilon\|_{\M(\R\times\T^3)}
 \leq T\|f\|_{\Xcal_T(\Om)}.
\end{equation}
\end{lemma}

\begin{proof}
Choose a nondecreasing function $\theta\in C^\infty(\R)$ such that
$0\leq\theta\leq1$, $\theta=0$ on $(-\infty,1/2]$, and $\theta=1$ on
$[1,\infty)$. For $0<\varepsilon<T/4$, set
\[
 \eta_\varepsilon(t)=\theta(t/\varepsilon)
 \theta((T-t)/\varepsilon).
\]
Then
\[
 \int_0^T|\eta_\varepsilon'|\dd t
 \leq2\|\theta'\|_{L^1(\R)}.
\]
Since $\eta_\varepsilon A$ vanishes in neighborhoods of both endpoints, its
zero extension produces no boundary Dirac masses. The product rule and
\eqref{eq:universal-tensor-div} give \eqref{eq:time-cutoff-div}; the
differentiated time column of $A$ is $(\rho,j)^{\mathsf T}$.

The first term in \eqref{eq:time-cutoff-div} is uniformly bounded in $L^1$
by \eqref{eq:rhoE-general-L1}. For the second term, mass conservation and
Cauchy--Schwarz give
\[
 \operatorname*{ess\,sup}_{0<t<T}\|\rho(t)\|_1=M,
 \qquad
 \operatorname*{ess\,sup}_{0<t<T}\|j(t)\|_1
 \leq M^{1/2}
 \left(\operatorname*{ess\,sup}_{0<t<T}M_2(t)\right)^{1/2}.
\]
These bounds are controlled by $\|f\|_{\Xcal_T(\Om)}$. Together with the
uniform $L^1$ bound on $\eta_\varepsilon'$, this proves
\eqref{eq:time-cutoff-div-bound}. If $\Om=\T^3$, then
\[
 \Tr_xB_\varepsilon
 =\eta_\varepsilon\tr\Pi
 =\eta_\varepsilon\int_{\R^3}|v|^2f\dd v,
\]
and hence
\[
 \|\Tr_xB_\varepsilon\|_{\M(\R\times\T^3)}
 \leq T\operatorname*{ess\,sup}_{0<t<T}M_2(t)
 \leq T\|f\|_{\Xcal_T(\Om)},
\]
which gives \eqref{eq:time-cutoff-trace-bound}.
\end{proof}

\begin{theorem}[Universal critical density estimate]
\label{thm:critical-density-universal}
Every weak solution $(f,E)$ on $[0,T]$ with $f\in\Xcal_T(\Om)$ obeys
\begin{equation}\label{eq:rho-L2-universal}
 \int_0^T\int_{\Om}\rho(t,x)^2\dd x\dd t
 \leq C_{\Om,T,\|f\|_{\Xcal_T(\Om)}}.
\end{equation}
More precisely, if $\|f\|_{\Xcal_T(\Om)}$ is bounded by a fixed constant,
the preceding argument gives a bound of order $C_{\Om}(1+T)^{4/3}$.
Consequently,
\begin{equation}\label{eq:E-critical-universal}
 E\in L^2(0,T;H^1(\Om))
 \hookrightarrow L^2(0,T;L^6(\Om)).
\end{equation}
\end{theorem}

\begin{proof}
If $\|f\|_{L^\infty((0,T)\times\Om\times\R^3)}=0$, then $f=0$
almost everywhere and the conclusion is trivial. Otherwise, apply
\eqref{eq:serre-unified} to $B_\varepsilon$. Since
$\det B_\varepsilon=\eta_\varepsilon^4\det A$,
\cref{lem:time-cutoff-divbv} yields
\[
 \int_0^T\int_{\Om}
 \eta_\varepsilon^{4/3}(\det A)^{1/3}\dd x\dd t
 \leq C_{\Om,T,\|f\|_{\Xcal_T(\Om)}}.
\]
Fatou's lemma and $\eta_\varepsilon\to1$ give the same estimate without the
cutoff. By \cref{lem:determinant}, applied pointwise in $(t,x)$,
\[
 \rho^2
 \leq C\|f\|_{L^\infty((0,T)\times\Om\times\R^3)}^{2/3}
 (\det A)^{1/3}
\]
almost everywhere. This proves \eqref{eq:rho-L2-universal}. Finally,
\cref{lem:poisson} gives $\nabla E\in L^2((0,T)\times\Om)$, while the
finite-energy estimates give $E\in L^\infty(0,T;L^2(\Om))$. This proves
\eqref{eq:E-critical-universal}.
\end{proof}

\begin{remark}[No approximation or endpoint trace is used]
\label{rem:no-approximation}
The proof of \cref{thm:critical-density-universal} uses only the given weak
solution. The cutoff term in \eqref{eq:time-cutoff-div} is controlled by mass
and the first moment, and the interior source is controlled before the
desired $L^2_{t,x}$ estimate. In particular, no convergence from classical
solutions and no strong trace of $(\rho,j)$ at $t=0,T$ enters the argument.
\end{remark}

\begin{remark}[Why the Serre applications are legitimate]
\label{rem:serre-check}
At the approximation level, the argument rests on four separate facts:
\begin{enumerate}[label=\textup{(\arabic*)},leftmargin=2.2em]
\item $A_n^e$ is a finite positive semidefinite tensor;
\item its row-divergence is the finite measure in
\eqref{eq:zero-extension-div};
\item the interior source $\rho_nE_n$ is in $L^1$ without using the desired
$L^2$ density estimate;
\item in the mixed periodic theorem, the additional trace is precisely the
local kinetic-energy density and has finite mass by the second-moment bound.
\end{enumerate}
For an arbitrary weak solution, the time-cutoff tensor $B_\varepsilon$
vanishes near both temporal endpoints, so its zero extension creates no
boundary Dirac masses; its divergence and periodic trace are controlled in
\eqref{eq:time-cutoff-div-bound}--\eqref{eq:time-cutoff-trace-bound}. No
divergence-free hypothesis is imposed in either argument, and the periodic
theorem is used on $\R_t\times\T_x^3$, not on a four-dimensional torus.
\end{remark}

\section{Moment propagation}\label{sec:moments}

\subsection{A common low-order estimate}

For moment propagation we use the uniform approximation-level form of the
critical density estimate in \cref{thm:critical-density-approx}.

We first treat $2<k\leq3$.  This is the new direct part of the moment
argument and works in both geometries and for either sign of the interaction.
For a classical approximation,
\begin{equation}\label{eq:Mk-derivative}
 \left|\frac{\dd}{\dd t}M_k(f_n(t))\right|
 \leq k\int_{\Om}|E_n(t,x)|m_{k-1,n}(t,x)\dd x.
\end{equation}
Taking $\ell=k-1$ in
\eqref{eq:integrated-kinetic-interpolation}, with
$p_k=(k+3)/(k+2)$, gives
\begin{equation}\label{eq:mk-pk}
 \|m_{k-1,n}\|_{p_k}
 \leq C_kF^{1/(k+3)}M_k(f_n)^{(k+2)/(k+3)}.
\end{equation}
The global estimate \eqref{eq:global-moment-interpolation} gives
\begin{equation}\label{eq:Mkminus1-global}
 \|m_{k-1,n}\|_1=M_{k-1}(f_n)
 \leq M^{1/k}M_k(f_n)^{(k-1)/k}.
\end{equation}
For $2<k\leq3$, define
\begin{equation*}
 \theta_k=\frac{3-k}{6}\in[0,1/6).
\end{equation*}
A direct computation gives
\begin{equation*}
 \frac{5}{6}=\theta_k+\frac{1-\theta_k}{p_k}.
\end{equation*}
Interpolating \eqref{eq:mk-pk} and
\eqref{eq:Mkminus1-global}, we obtain the geometry-independent estimate
\begin{equation}\label{eq:mk-65}
 \|m_{k-1,n}\|_{6/5}
 \leq C_kF^{1/6}M^{(3-k)/(6k)}
 M_k(f_n)^{1-1/(2k)}.
\end{equation}
Notice that this argument does not use the finite measure of $\T^3$ and
therefore applies without change in $\R^3$.

Using H\"older's inequality in \eqref{eq:Mk-derivative} and
\eqref{eq:mk-65},
\begin{equation*}
 \frac{\dd}{\dd t}M_k(f_n(t))
 \leq C_kF^{1/6}M^{(3-k)/(6k)}
 M_k(f_n(t))^{1-1/(2k)}\|E_n(t)\|_6.
\end{equation*}
Set
\[
 Z_{k,n}(t)=(1+M_k(f_n(t)))^{1/(2k)}.
\]
Then
\begin{equation}\label{eq:Zk-ineq}
 Z_{k,n}'(t)
 \leq C_kF^{1/6}M^{(3-k)/(6k)}\|E_n(t)\|_6.
\end{equation}
By \cref{thm:critical-density-approx},
$E_n$ is uniformly bounded in $L^2(0,T;L^6)$, and hence in
$L^1(0,T;L^6)$.  Integrating \eqref{eq:Zk-ineq} gives
\begin{equation}\label{eq:low-moment-uniform}
 \sup_n\sup_{0\leq t\leq T}M_k(f_n(t))<\infty,
 \qquad2<k\leq3.
\end{equation}

\begin{remark}[The endpoint $k=3$]\label{rem:k3}
At $k=3$, $p_k=6/5$ and $\theta_k=0$, so
\eqref{eq:mk-65} is exactly the standard $L^6$--$L^{6/5}$ pairing.  There is
no endpoint loss.  For $k>3$, one has $p_k<6/5$, and the present interpolation
cannot place $m_{k-1}$ in $L^{6/5}$; this is the genuine boundary of this
short argument.
\end{remark}

\subsection{Passage to weak solutions and completion of the range}

By \cref{lem:time-slice} and \eqref{eq:low-moment-uniform},
\begin{equation}\label{eq:low-moment-limit}
 \sup_{0\leq t\leq T}M_k(t)<\infty,
 \qquad2<k\leq3.
\end{equation}
This proves the new low-order assertion in both $\R^3$ and $\T^3$.

For $k>3$, we use the known propagation estimates. In the whole space,
Pallard proved propagation for every $k>2$ \cite{Pallard2012}; in the periodic problem,
Chen--Chen proved it for $k>3$
\cite{ChenChen2019}. Applying the a priori estimate
\cite[Theorem~3.4]{ChenChen2019} to the smooth periodic solution $f_n$ gives
\[
 \sup_{0\leq t\leq T}M_k(f_n(t))
 \leq
 C\left(T,k,M,F,\Kstar,\sup_n M_k(f_{0,n})\right),
\]
where $C$ is independent of $n$. Since
$M_k(f_{0,n})\to M_k(f_0)$, the right-hand side is finite.
A second application of \cref{lem:time-slice} passes the estimate to the
limit.  We have proved the moment part of \cref{thm:moment-main}.

\begin{corollary}[Low-order moment propagation for both signs]
\label{cor:low-both-signs}
Let $\Om\in\{\R^3,\T^3\}$, $\gamma\in\{-1,1\}$, and $2<k\leq3$.
Assume
$0\leq f_0\in L^1(\Om\times\R^3)\cap L^\infty(\Om\times\R^3)$ and
$M_k(0)<\infty$. Then the finite-energy approximation constructed in
\cref{prop:compactness} has a subsequence whose limit satisfies
\eqref{eq:low-moment-limit} in either geometry.
\end{corollary}

\begin{proof}
The sign does not enter \eqref{eq:Mk-derivative}--\eqref{eq:Zk-ineq}; it
enters only the uniform second-moment estimate, which is available for both
signs by \cref{lem:uniform-K}.
\end{proof}

\section{Energy conservation}
\label{sec:energy}

The energy argument is intrinsic. We first isolate the basic mechanism under
strong time continuity, retaining the criterion useful for the constructed
solutions. We then remove the strong time-continuity assumption and obtain
the universal identity stated in \cref{thm:energy-universal}.

\subsection{Energy equality under strong time continuity}

\begin{proposition}[Critical energy-equality criterion]
\label{prop:energy-criterion}
Let $\Om\in\{\R^3,\T^3\}$ and $\gamma\in\{-1,1\}$. Let $(f,E)$ be a
nonnegative weak solution on $[0,T]$ satisfying
\[
 f\in C([0,T];L^1(\Om\times\R^3))\cap\Xcal_T(\Om),
 \qquad M_2(0)<\infty,
\]
and
\begin{equation}\label{eq:energy-criterion-rho}
 \rho\in L^2((0,T)\times\Om).
\end{equation}
Then $\Kcal$ and $\Pcal$ in \eqref{eq:energy-def} have absolutely continuous
representatives and, for every $0\leq s\leq t\leq T$,
\begin{align}
 \Kcal(t)-\Kcal(s)&=\int_s^t\int_{\Om}E(\tau,x)\cdot j(\tau,x)\dd x\dd\tau,
 \label{eq:kinetic-balance}\\
 \Pcal(t)-\Pcal(s)&=-\int_s^t\int_{\Om}E(\tau,x)\cdot j(\tau,x)\dd x\dd\tau.
 \label{eq:potential-balance}
\end{align}
Consequently, $\Ecal(t)=\Ecal(s)$.
\end{proposition}

\begin{proof}
\emph{Step 1: integrability of the power density.}
Put
\[
 m_2(t,x)=\int_{\R^3}|v|^2f(t,x,v)\dd v,
 \qquad
 J(t,x)=\int_{\R^3}|v|f(t,x,v)\dd v.
\]
Taking $(k,\ell)=(2,1)$ in
\eqref{eq:integrated-kinetic-interpolation}, and using
\eqref{eq:global-moment-interpolation}, we obtain
\begin{equation}\label{eq:J-bounds-energy}
 J\in L^\infty(0,T;L^1(\Om)\cap L^{5/4}(\Om)),
 \qquad
 j\in L^\infty(0,T;L^{6/5}(\Om)).
\end{equation}
The last assertion follows by interpolation between $L^1$ and $L^{5/4}$.
Taking $(k,\ell)=(2,0)$ in
\eqref{eq:integrated-kinetic-interpolation} gives
$\rho\in L^\infty_t(L^1\cap L^{5/3})$, hence
$E\in L^\infty_tL^2_x$.  By \eqref{eq:energy-criterion-rho} and
\cref{lem:poisson},
\begin{equation}\label{eq:E-H1-energy}
 E\in L^2(0,T;H^1(\Om))
 \hookrightarrow L^2(0,T;L^6(\Om)).
\end{equation}
Therefore
\[
 E\cdot j\in L^1((0,T)\times\Om).
\]

\emph{Step 2: a truncated kinetic-energy identity.}
Let $\eta\in C_c^\infty((0,T))$. Choose $\chi\in C_c^\infty([0,\infty))$ such that
\[
 0\leq\chi\leq1,
 \qquad \chi=1\text{ on }[0,1],
 \qquad \chi=0\text{ on }[2,\infty),
 \qquad \chi'\leq0,
\]
and define
\begin{equation*}
 \beta_R(v)=\frac{|v|^2}{2}\chi\left(\frac{|v|}{R}\right).
\end{equation*}
For fixed $v$, $\beta_R(v)$ increases to $|v|^2/2$ as $R\to\infty$, and
\begin{equation}\label{eq:beta-R-properties}
 |\nabla_v\beta_R(v)|\leq C|v|,
 \qquad
 |\nabla_v\beta_R(v)-v|
 \leq C|v|\one_{\{|v|>R\}}.
\end{equation}
Set
\begin{equation*}
 \Kcal_R(t)=\iint\beta_R(v)f(t,x,v)\dd x\dd v,
 \qquad
 j_R(t,x)=\int\nabla_v\beta_R(v)f(t,x,v)\dd v.
\end{equation*}

In the periodic case, use the test function $\eta(t)\beta_R(v)$ directly in
\eqref{eq:weak-form}.  In the Euclidean case, multiply it by a spatial cutoff
$\zeta_L$, equal to one on $B_L$, supported in $B_{2L}$, and satisfying
$|\nabla\zeta_L|\leq C/L$.  For fixed $R$, the spatial boundary term is bounded by
\[
\begin{aligned}
 \frac{C}{L}
 \iint_{\Om\times\R^3}|v|\beta_R(v)f(t,x,v)\,\dd x\,\dd v
 &\leq
 \frac{CR}{L}M_2(t).
\end{aligned}
\]
After integration in time, it is bounded by
\[
 \frac{CR}{L}\|\eta\|_{L^1(0,T)}
 \operatorname*{ess\,sup}_{0<t<T}M_2(t),
\]
and therefore vanishes as $L\to\infty$.  The force term is dominated by \(C|E|J\), which belongs to
\(L^1((0,T)\times\Om)\)  by \eqref{eq:J-bounds-energy} and
\eqref{eq:E-H1-energy}.  We obtain
\begin{equation*}
 -\int_0^T\Kcal_R(t)\eta'(t)\dd t
 =\int_0^T\eta(t)\int_{\Om}E(t,x)\cdot j_R(t,x)\dd x\dd t.
\end{equation*}
Because $\beta_R$ is bounded and $f\in C_tL^1$, $\Kcal_R$ is continuous.  The
right-hand side belongs to $L^1(0,T)$, so
\begin{equation}\label{eq:KR-balance-energy}
 \Kcal_R(t)-\Kcal_R(s)
 =\int_s^t\int_{\Om}E(\tau,x)\cdot j_R(\tau,x)\dd x\dd\tau
\end{equation}
for every $0\leq s\leq t\leq T$.

\emph{Step 3: removal of the velocity truncation.}
Define
\[
 q_R(t,x)=\int_{|v|>R}|v|f(t,x,v)\dd v.
\]
By \eqref{eq:beta-R-properties}, $|j_R-j|\leq Cq_R$.  The second moment gives
\begin{equation*}
 \|q_R(t)\|_1
 \leq R^{-1}M_2(t)\leq CR^{-1}
\end{equation*}
for almost every $t$.  Since $q_R\leq J$, it is uniformly bounded in
$L^{5/4}$.  Interpolation between $L^1$ and $L^{5/4}$ gives
\begin{equation*}
 \|q_R\|_{L^\infty(0,T;L^{6/5})}\leq CR^{-1/6},
\end{equation*}
and therefore
\begin{equation*}
 j_R\longrightarrow j
 \quad\text{strongly in }L^\infty(0,T;L^{6/5}(\Om)).
\end{equation*}
Since $E\in L^1(0,T;L^6)$,
\[
 \sup_{0\leq t\leq T}
 \left|\int_0^t\int_{\Om}E\cdot(j_R-j)\dd x\dd\tau\right|
 \longrightarrow0.
\]
By monotone convergence,
\[
 \mathcal K_R(\tau)\uparrow
 \frac12\iint_{\Om\times\R^3}|v|^2
 f(\tau,x,v)\,\dd x\,\dd v
 =:\mathcal K(\tau)
\]
for every \(\tau\in[0,T]\), with the limit initially allowed to be
infinite. Taking \(s=0\) in \eqref{eq:KR-balance-energy},
using \(\mathcal K_R(0)\uparrow\mathcal K(0)<\infty\), and passing to
the limit in the right-hand side, we obtain
\[
 \mathcal K(t)
 =\mathcal K(0)
 +\int_0^t\int_\Om E\cdot j\,\dd x\,\dd\tau<\infty.
\]
This gives the absolutely continuous representative of $\Kcal$.
Subtracting the same identity at times $t$ and $s$ proves
\eqref{eq:kinetic-balance}.

\emph{Step 4: the potential-energy identity.}
The continuity equation in \cref{prop:macroscopic-balances}
implies \eqref{eq:E-time-weak}. Hence, by \eqref{eq:J-bounds-energy},
\begin{equation*}
 \partial_tE=-\lamOm\gamma Q_{\Om}j
 \in L^2(0,T;H^{-1}(\Om)).
\end{equation*}
Together with \eqref{eq:E-H1-energy}, the Lions--Magenes lemma
\cite{LionsMagenes1972} gives
\begin{equation*}
 E\in C([0,T];L^2(\Om)),
 \qquad
 \frac12\frac{\dd}{\dd t}\|E(t)\|_2^2
 =\langle\partial_tE,E\rangle_{H^{-1},H^1}.
\end{equation*}
The Calder\'on--Zygmund projection $Q_{\Om}$ is symmetric with respect to
the $L^{6/5}$--$L^6$ dual pairing, so that
$\langle Q_{\Om}j,E\rangle=\langle j,Q_{\Om}E\rangle$. Moreover,
$Q_{\Om}E=E$, because $E$ is a gradient field
(and has zero mean on the torus).  Hence, using $\gamma^2=1$,
\begin{align}
 \frac{\dd}{\dd t}\Pcal(t)
 &=\frac{\gamma}{\lamOm}\langle\partial_tE,E\rangle\notag\\
 &=-\langle Q_{\Om}j,E\rangle
 =-\int_{\Om}E(t,x)\cdot j(t,x)\dd x.
 \label{eq:P-derivative-energy}
\end{align}
It remains to identify the continuous value at the initial time. Since
$f(t)\to f_0$ in $L^1$, one has $\rho(t)\to\rho_0$ in distributions and,
in the periodic case, $\bar\rho(t)=\bar\rho_0$. Passing to the limit in
\[
 \diver E(t)=\lamOm\gamma\rhos(t),
 \qquad
 \nabla\times E(t)=0,
\]
and using $E(t)\to E(0)$ strongly in $L^2$, we conclude that $E(0)$ is
precisely the Poisson field generated by $\rho_0$, with the normalization
specified after \eqref{eq:VP-system}. Integrating
\eqref{eq:P-derivative-energy} proves \eqref{eq:potential-balance}. Adding
the two balances proves conservation of $\Ecal$.
\end{proof}

\subsection{The universal energy identity}

We now remove the assumption $f\in C([0,T];L^1)$ from
\cref{prop:energy-criterion}. The initial term in the weak formulation
provides the correct trace for every bounded velocity truncation.

\begin{theorem}[Universal energy identity]
\label{thm:energy-identity-universal}
Let $\Om\in\{\R^3,\T^3\}$ and $\gamma\in\{-1,1\}$. Let $(f,E)$ be a
nonnegative weak solution on $[0,T]$ with $f\in\Xcal_T(\Om)$, and assume
\[
 M_2(0)=\iint_{\Om\times\R^3}|v|^2f_0(x,v)\dd x\dd v<\infty.
\]
Then $\Kcal$ and $\Pcal$ in \eqref{eq:energy-def} admit absolutely continuous
representatives on $[0,T]$, and for every $0\leq s\leq t\leq T$,
\begin{align}
 \Kcal(t)-\Kcal(s)
 &=\int_s^t\int_{\Om}E(\tau,x)\cdot j(\tau,x)\dd x\dd\tau,
 \label{eq:kinetic-balance-universal}\\
 \Pcal(t)-\Pcal(s)
 &=-\int_s^t\int_{\Om}E(\tau,x)\cdot j(\tau,x)\dd x\dd\tau.
 \label{eq:potential-balance-universal}
\end{align}
Consequently, $\Ecal(t)=\Ecal(s)$. No renormalization or approximation
hypothesis is used.
\end{theorem}

\begin{proof}
\emph{Step 1: integrability of the power density.}
Let
\[
 J(t,x)=\int_{\R^3}|v|f(t,x,v)\dd v.
\]
By \cref{lem:density-current,prop:macroscopic-balances},
\[
 J\in L^\infty(0,T;L^1(\Om)\cap L^{5/4}(\Om)),
 \qquad
 j\in L^\infty(0,T;L^{6/5}(\Om)).
\]
By \cref{thm:critical-density-universal},
$E\in L^2(0,T;L^6(\Om))$. Therefore
\begin{equation}\label{eq:power-universal}
 E\cdot j\in L^1((0,T)\times\Om).
\end{equation}

\emph{Step 2: truncated kinetic energy and its initial trace.}
Choose $\chi\in C_c^\infty([0,\infty))$ such that
\[
 0\leq\chi\leq1,
 \qquad \chi=1\text{ on }[0,1],
 \qquad \chi=0\text{ on }[2,\infty),
 \qquad \chi'\leq0,
\]
and define
\[
 \beta_R(v)=\frac{|v|^2}{2}\chi\left(\frac{|v|}{R}\right),
 \qquad
 \Kcal_R(t)=\iint\beta_R(v)f(t,x,v)\dd x\dd v,
\]
\[
 j_R(t,x)=\int_{\R^3}\nabla_v\beta_R(v)f(t,x,v)\dd v,
 \qquad
 \Kcal_R^0=\iint\beta_R(v)f_0(x,v)\dd x\dd v.
\]
Then $\beta_R\uparrow|v|^2/2$ pointwise and
\begin{equation}\label{eq:beta-R-universal}
 |\nabla_v\beta_R(v)|\leq C|v|,
 \qquad
 |\nabla_v\beta_R(v)-v|
 \leq C|v|\one_{\{|v|>R\}}.
\end{equation}

On $\T^3$, use $\eta(t)\beta_R(v)$ directly in \eqref{eq:weak-form}. On
$\R^3$, multiply it by a spatial cutoff $\zeta_L$, equal to one on $B_L$,
supported in $B_{2L}$, and satisfying $|\nabla\zeta_L|\leq C/L$. Since
$\beta_R$ is supported in $\{|v|\leq2R\}$,
\[
 \frac1L\int_0^T\iint |v|\beta_R(v)f\dd x\dd v\dd t
 \leq \frac{CRT}{L}
 \operatorname*{ess\,sup}_{0<t<T}\iint |v|^2f(t)\dd x\dd v,
\]
so the spatial boundary term vanishes as $L\to\infty$. We obtain, for every
$\eta\in C_c^\infty([0,T))$,
\begin{equation}\label{eq:KR-distribution-universal}
 -\int_0^T\Kcal_R(t)\eta'(t)\dd t
 =\Kcal_R^0\eta(0)
 +\int_0^T\eta(t)\int_{\Om}E(t,x)\cdot j_R(t,x)\dd x\dd t.
\end{equation}
Indeed, $|j_R|\leq CJ$, and interpolation gives
$J\in L^\infty_tL^{6/5}_x$, so the right-hand side is integrable. The
one-dimensional distributional fundamental theorem therefore gives an
absolutely continuous representative satisfying
\begin{equation}\label{eq:KR-balance-universal}
 \Kcal_R(t)=\Kcal_R^0
 +\int_0^t\int_{\Om}E(\tau,x)\cdot j_R(\tau,x)\dd x\dd\tau,
 \qquad 0\leq t\leq T.
\end{equation}
This step uses the initial term in the weak formulation and does not assume
$f\in C([0,T];L^1)$.

\emph{Step 3: removal of the velocity cutoff.}
By \eqref{eq:beta-R-universal},
\[
 |j_R-j|\leq Cq_R.
\]
The second moment gives
\[
 \|q_R\|_{L^\infty(0,T;L^1(\Om))}
 \leq R^{-1}\operatorname*{ess\,sup}_{0<t<T}M_2(t),
\]
and $q_R\leq J$ is uniformly bounded in
$L^\infty(0,T;L^{5/4}(\Om))$. Interpolation therefore yields
\[
 \|q_R\|_{L^\infty(0,T;L^{6/5}(\Om))}
 \leq C_{\|f\|_{\Xcal_T(\Om)}}R^{-1/6}.
\]
Consequently,
\begin{equation}\label{eq:jR-universal-convergence}
 j_R\longrightarrow j
 \quad\text{strongly in }L^\infty(0,T;L^{6/5}(\Om)).
\end{equation}
Since $E\in L^1(0,T;L^6(\Om))$, the integrals on the right-hand side of
\eqref{eq:KR-balance-universal} converge uniformly with respect to the upper
time endpoint. Moreover, $\Kcal_R^0\uparrow\Kcal(0)$. Thus the absolutely
continuous representatives converge uniformly on $[0,T]$ to
\begin{equation}\label{eq:K-sharp}
 \Kcal^\sharp(t)
 :=\Kcal(0)+\int_0^t\int_{\Om}E(\tau,x)\cdot j(\tau,x)\dd x\dd\tau.
\end{equation}
Choose an increasing sequence $R_n\to\infty$. For almost every $t$, all
phase-space integrals defining $\Kcal_{R_n}(t)$ agree with their absolutely
continuous representatives, and monotone convergence gives
\[
 \Kcal^\sharp(t)
 =\frac12\iint_{\Om\times\R^3}|v|^2f(t,x,v)\dd x\dd v.
\]
We take $\Kcal^\sharp$ as the canonical representative of the kinetic
energy. Subtracting \eqref{eq:K-sharp} at $s$ and $t$ proves
\eqref{eq:kinetic-balance-universal}.

\emph{Step 4: potential energy.}
By \eqref{eq:E-time-weak} and the current bound,
\[
 \partial_tE=-\lamOm\gamma Q_{\Om}j
 \in L^2(0,T;H^{-1}(\Om)).
\]
Together with \eqref{eq:E-critical-universal}, the Lions--Magenes lemma
\cite{LionsMagenes1972} gives
\[
 E\in C([0,T];L^2(\Om)),
 \qquad
 \frac12\frac{\dd}{\dd t}\|E(t)\|_2^2
 =\langle\partial_tE,E\rangle_{H^{-1},H^1}
\]
for almost every $t$. The projection $Q_{\Om}$ is symmetric in the
$L^{6/5}$--$L^6$ duality and $Q_{\Om}E=E$. Hence
\begin{equation}\label{eq:P-derivative-universal}
 \frac{\dd}{\dd t}\Pcal(t)
 =\frac{\gamma}{\lamOm}\langle\partial_tE,E\rangle
 =-\int_{\Om}E(t,x)\cdot j(t,x)\dd x.
\end{equation}

It remains to identify the value at $t=0$. By
\cref{prop:macroscopic-balances},
$\rho\in C([0,T];\mathcal D'(\Om))$ and $\rho(0)=\rho_0$. Choose Lebesgue
times $t_n\downarrow0$ at which the Poisson equations hold. Since
$E(t_n)\to E(0)$ strongly in $L^2$ and $\rho(t_n)\to\rho_0$ in
distributions, $E(0)$ is precisely the normalized Poisson field generated
by $\rho_0$. Integrating \eqref{eq:P-derivative-universal} proves
\eqref{eq:potential-balance-universal}. Adding the two balances proves
conservation of $\Ecal$.
\end{proof}

\subsection{Proofs of the main results}

\begin{proof}[Proof of \cref{thm:energy-universal}]
Fix $T>0$. \Cref{prop:macroscopic-balances} gives the macroscopic balance
laws, \cref{thm:critical-density-universal} gives the critical density and
field regularity, and \cref{thm:energy-identity-universal} gives the two
energy balances. The representatives constructed on different finite
intervals coincide on their overlaps, since they are given by the same
integral identities with initial time $0$. Since $T$ is arbitrary,
\eqref{eq:energy-main-conservation} holds globally.
\end{proof}

\begin{proof}[Proof of \cref{cor:energy-existence}]
\Cref{prop:compactness} supplies a global nonnegative renormalized weak
solution, obtained as a limit of classical solutions, with strong time
continuity, an $L^\infty$ bound, and a uniformly bounded second moment on
every finite time interval. It therefore satisfies the assumptions of
\cref{thm:energy-universal}, which proves the conclusion.
\end{proof}

\begin{proof}[Completion of the proof of \cref{thm:moment-main}]
The moment estimate was proved in \cref{sec:moments}. Since $k>2$ implies
$M_2(f_0)\leq M+M_k(f_0)$, the constructed solution lies in the
finite-energy framework of \cref{prop:compactness}. The regularity stated in
\cref{rem:common-regularity} follows from \cref{prop:compactness},
\cref{thm:critical-density-approx}, and \cref{thm:energy-universal}; the
latter also shows that this moment-propagating solution conserves the total
energy.
\end{proof}

\begin{remark}[What is critical at $k=2$]\label{rem:critical-final}
The density estimate supplies precisely
$E\in L^2_tL^6_x$, while the second moment supplies precisely
$j\in L^\infty_tL^{6/5}_x$.  Thus the power density is integrable at the
dual Sobolev exponents.  The kinetic truncation needs to remove only the
first-order tail $\int_{|v|>R}|v|f\dd v$, for which the second moment gives
the quantitative $R^{-1/6}$ decay in $L^{6/5}_x$.  No moment strictly above
order two enters the energy proof.
\end{remark}

\paragraph{Data/Code Availability.}
The manuscript has no associated data or code.

\paragraph{Conflict of interest.}
The author declares that he has no conflict of interest.

\end{document}